\documentclass[11pt,twoside,a4paper]{article}

\usepackage{amssymb}
\usepackage{amsmath}
\usepackage{amsthm}
\usepackage{color}
\usepackage{mathrsfs}

\usepackage{amsfonts}
\usepackage{amssymb}
\usepackage{amsmath}
\usepackage{amsthm}
\usepackage{bbm}
\usepackage{verbatim}
\usepackage{url}

\allowdisplaybreaks

\def\Z{{\mathbb Z}}

\newcommand{\eps}[0]{\varepsilon}

\def\fz{\infty}

\def\ls{\lesssim}
\def\gs{\gtrsim}

\def\dint{\displaystyle\int}

\def\dfrac{\displaystyle\frac}

\def\r{\right}
\def\lf{\left}

\newtheorem{thm}{Theorem}[section]
\newtheorem{lem}[thm]{Lemma}
\newtheorem{prop}[thm]{Proposition}
\newtheorem{defn}[thm]{Definition}

\newcommand{\R}{\mathbb R}

\newcommand{\intav}{-\!\!\!\!\!\!\int}

\DeclareMathOperator*{\essinf}{ess\ inf}
\DeclareMathOperator*{\esssup}{ess\ sup}

\numberwithin{equation}{section}

\begin{document}

\arraycolsep=1pt

\title{\Large\bf A note on two weight commutators of maximal  functions\\ on spaces of homogeneous type}
\author{ Ruming Gong, Manasa N. Vempati and Qingyan Wu}

\date{}
\maketitle

\begin{center}
\begin{minipage}{13.5cm}\small

{\noindent  {\bf Abstract:}\
We study the two weight quantitative estimates for the commutator of maximal functions and the maximal commutators with respect to the symbol in weighted BMO space on spaces of homogeneous type. These commutators turn out to be controlled by the sparse operators in the setting of space of homogeneous type (developed in \cite{DGKLWY}, originally introduced in \cite{LOR}). The lower bound of the maximal commutator is also obtained.
}

\end{minipage}
\end{center}

%
%
\bigskip


\medskip


\section{Introduction and statement of main results\label{s1}}

In their remarkable result, Coifman--Rochberg--Weiss \cite{crw} showed that the commutator of Riesz transforms
is bounded on $L^p(\mathbb R^n)$ if and only if the symbol $b$ is in the BMO space. See also the subsequent result by Janson \cite{Ja} and Uchiyama \cite{u81}. Later, Bloom \cite{B} obtained the two weight version of the commutator of Hilbert transform $H$ with respect to weighted BMO space. To be more precise, for
$1 < p < \infty$, let $\lambda_{1}, \lambda_{2}$ be weights in the Muckenhoupt class $A_p$ and consider the weight $\nu = \lambda_{1}^{1/p} \lambda_{2}^{-1/p}$. Let $L^p_{w}(\R)$ denote the space of functions that are $p$ integrable relative to the measure $w(x)dx$.  Then, by \cite{B}, there
exist constants $0 < c < C < \infty$, depending only on $p, \lambda_{1}, \lambda_{2}$, such that
\begin{equation*}
c\| b \| _{{\rm BMO}_\nu(\R)} \le  \| [b, H] : L^p_{\lambda_{1}}(\R) \rightarrow L^p_{\lambda_{2}}(\R) \|  \le C \| b \| _{{\rm BMO}_\nu(\R)}
\end{equation*}
in which $[b, H] (f)(x) = b(x) H (f)(x) - H (bf)(x)$ denotes the commutator of the Hilbert transform $H$ and the function $b \in {\rm BMO}_\nu(\R)$, i.e., the Muckenhoupt--Wheeden weighted BMO space (introduced in \cite{MW76}, see also the definition in Section 2.4 below).  This result provided a characterization of the boundedness of the commutator $[b, H]:L^p_{\lambda_{1}}(\R ) \rightarrow L^p_{\lambda_{2}}(\R ) $ in terms of a triple of information $b,\lambda_{1}$ and $\lambda_{2}$.  This result was extended very recently to the commutator of Riesz transform in $\R^n$ by Holmes--Lacey--wick \cite{HLW} by using different method via representation theorem for the Riesz transforms. Lerner--Ombrosi--Rivera-R\'ios also  proved it in \cite{LOR} by using sparse domination and their method was generalised to space of homogeneous type in \cite{DGKLWY}.

In \cite[Propositions 4 and 6]{BMR},  Bastero--Milman--Ruiz characterized the class of functions for which the commutator with the Hardy--Littlewood maximal function and  the maximal sharp function are bounded on $L^p$. Later, Garc\'ia-Cuerva et al. \cite[Theorem 2.4]{GHST} proved that the maximal commutator $ {C}_b$ is bounded from $L^p_{\lambda_1}(\mathbb R^n)$ to $L^p_{\lambda_2}(\mathbb R^n)$, $1<p<\infty$,
 if and only if $b\in {\rm BMO}_\nu(\mathbb R^n)$ with $\nu = \lambda_{1}^{1/p} \lambda_{2}^{-1/p}$,
where $ {C}_b$ is defined by
$$ {C}_b(f)(x)=\sup_{Q\ni x}{1\over |Q|}\int_Q|b(x)-b(y)| |f(y)|dy.$$
Hu--Yang \cite{HY} also studied the unweighted upper bound of the maximal commutator $ {C}_b(f)$ on spaces of homogeneous type by adapting the approach in \cite{GHST}. Recently,  Agcayazi et al \cite{AGKM} also studied the
unweighted version of the maximal commutator $ {C}_b(f)$ on $\mathbb R^n$ by using different approach, and this was extended to space of homogeneous type by Fu et al \cite{FPW}.

In this paper, we aim to provide a quantitative estimate for the two weight commutator of maximal functions $[b,\mathcal M]$ and the maximal commutator $C_b$ with the symbol $b$ in weighted BMO space on spaces of homogeneous type. To be more precise, let $(X,d,\mu)$ be a space of homogeneous type in the sense of Coifman and Weiss \cite{cw77} (see the definition and details in Section 2 below).
The Hardy--Littlewood maximal function $\mathcal Mf(x)$ on $X$ is defined as
$$ \mathcal Mf(x):=\sup_{B \ni x} {1\over \mu(B)}\int_B |f(y)|\,d\mu(y), $$
where the supremum is taken over all balls $B\subset X$.
The maximal commutator $C_b$ on $X$ with the symbol $b(x)$ is defined by
$$C_b(f)(x):=\sup_{B\ni x}{1\over\mu(B)}\int_{B}|b(x)-b(y)| |f(y)|d\mu(y),$$
where the supremum is taken over all balls $B\subset X$.

Our first result is   the quantitative estimate of $[b,\mathcal M]$.
\begin{thm}\label{thm main1}
Suppose $1<p<\infty$, $\lambda_1,\lambda_2\in A_p$, $\nu:= \lambda_1^{1\over p}\lambda_2^{-{1\over p}}$. Suppose $b\in {\rm BMO}_{\nu}(X)$. Then
there exists a positive constant $C$ such that
\begin{equation*}
\|  [b, \mathcal M] : L^p_{\lambda_1}(X) \rightarrow L^p_{\lambda_2}(X) \|  \le C\Big([\lambda_1]_{A_p} [\lambda_2]_{A_p}  \Big)^{\max\{1,{1\over p-1}\}}  \| b \| _{{\rm BMO}_{\nu}(X)}  .
\end{equation*}
\end{thm}

Note that $[b, \mathcal M](f)(x)$ is dominated by $C_b(f)(x)$. To prove the above result, it suffices to show that
\begin{thm}\label{thm main2}
Suppose $1<p<\infty$, $\lambda_1,\lambda_2\in A_p$, $\nu:= \lambda_1^{1\over p}\lambda_2^{-{1\over p}}$. Suppose $b\in {\rm BMO}_{\nu}(X)$. Then
there exists a positive constant $C$ such that
\begin{equation*}
\|  {C}_b : L^p_{\lambda_1}(X) \rightarrow L^p_{\lambda_2}(X) \|  \le C\Big([\lambda_1]_{A_p} [\lambda_2]_{A_p}  \Big)^{\max\{1,{1\over p-1}\}}  \| b \| _{{\rm BMO}_{\nu}(X)}  .
\end{equation*}
\end{thm}
We note that the approach we use is via sparse domination of $C_b$ (see Section 3.1), and hence we obtain a better quantitative estimate with respect to the weights $\lambda_1$ and $\lambda_2$ comparing to the methods used in \cite{GHST} and \cite{AGKM}. We will explain this in Section 3.2.

We point out that the lower bound of $\|  {C}_b : L^p_{\lambda_1}(X) \rightarrow L^p_{\lambda_2}(X) \|$ is
also true.
\begin{thm}\label{thm main3}
Suppose $1<p<\infty$, $\lambda_1,\lambda_2\in A_p$, $\nu:= \lambda_1^{1\over p}\lambda_2^{-{1\over p}}$. Suppose $b\in L^1_{\rm loc}(X)$ and that   $ {C}_b$ is bounded from
$ L^p_{\lambda_1}(X) $ to $ L^p_{\lambda_2}(X)$.
Then $b\in {\rm BMO}_{\nu }(X)$, and there exists a positive constant $C $ such that
\begin{equation*}
 \| b \| _{{\rm BMO}_{\nu }(X)} \leq C  \|  {C}_b : L^p_{\lambda_1}(X) \rightarrow L^p_{\lambda_2}(X) \|.
\end{equation*}
\end{thm}
We will provide the proof in Section 4.

Throughout this paper we assume that $\mu(X)=\infty$ and that $\mu(\{x_0\})=0$ for every $x_0\in X$.
Also we denote by $C$ and $\widetilde{C}$ {\it positive constants} which
are independent of the main parameters, but they may vary from line to
line. For every $p\in(1, \fz)$, we denote by $p'$ the conjugate of $p$, i.e., $\frac{1}{p'}+\frac{1}{p}=1$.  If $f\le Cg$ or $f\ge Cg$, we then write $f\ls g$ or $g\gs f$;
and if $f \ls g\ls f$, we  write $f\approx g.$

\section{Preliminaries on Spaces of Homogeneous Type}
\label{s2}
\noindent

We say
that $(X,d,\mu)$ is a {space of homogeneous type} in the
sense of Coifman and Weiss if $d$ is a quasi-metric on~$X$
and $\mu$ is a nonzero measure satisfying the doubling
condition. A \emph{quasi-metric}~$d$ on a set~$X$ is a
function $d: X\times X\longrightarrow[0,\infty)$ satisfying
(i) $d(x,y) = d(y,x) \geq 0$ for all $x$, $y\in X$; (ii)
$d(x,y) = 0$ if and only if $x = y$; and (iii) the
\emph{quasi-triangle inequality}: there is a constant $A_0\in
[1,\infty)$ such that for all $x$, $y$, $z\in X$, 
\begin{eqnarray}\label{eqn:quasitriangleineq}
    d(x,y)
    \leq A_0 [d(x,z) + d(z,y)].
\end{eqnarray}
We say that a nonzero measure $\mu$ satisfies the
\emph{doubling condition} if there is a constant $C_\mu$ such
that for all $x\in X$ and $r > 0$,
\begin{eqnarray}\label{doubling condition}
   \mu(B(x,2r))
   \leq C_\mu \mu(B(x,r))
   < \infty,
\end{eqnarray}
where $B(x,r)$ is the quasi-metric ball by $B(x,r) := \{y\in X: d(x,y)
< r\}$ for $x\in X$ and $r > 0$.We point out that the doubling condition (\ref{doubling
condition}) implies that there exists a positive constant
$n$ (the \emph{upper dimension} of~$\mu$)  such
that for all $x\in X$, $\lambda\geq 1$ and $r > 0$,
\begin{eqnarray}\label{upper dimension}
    \mu(B(x, \lambda r))
    \leq  C_\mu\lambda^{n} \mu(B(x,r)).
\end{eqnarray}


A subset $\Omega\subseteq X$ is \emph{open} (in the topology
induced by $\rho$) if for every $x\in\Omega$ there exists
$\eps>0$ such that $B(x,\eps)\subseteq\Omega$. A subset
$F\subseteq X$ is \emph{closed} if its complement $X\setminus
F$ is open. The usual proof of the fact that $F\subseteq X$ is
closed, if and only if it contains its limit points, carries
over to the quasi-metric spaces. However, some open balls
$B(x,r)$ may fail to be open sets, see \cite[Sec 2.1]{HK}.

Constants that depend only on $A_0$ (the quasi-metric constant)
and $A_1$ (the geometric doubling constant), are referred to as
\textit{geometric constants.}

\subsection{A System of Dyadic Cubes}\label{sec:dyadic_cubes}

We recall from \cite{HK} (see also the previous work by M.~Christ \cite{Chr}, as
well as Sawyer--Wheeden~\cite{SawW}) the system of dyadic cubes.
In a geometrically doubling quasi-metric space $(X,d)$, a
countable family
\[
    \mathscr{D}
    = \bigcup_{k\in\Z}\mathscr{D}_k, \quad
    \mathscr{D}_k
    =\{Q^k_\alpha\colon \alpha\in \mathscr{A}_k\},
\]
of Borel sets $Q^k_\alpha\subseteq X$ is called \textit{a
system of dyadic cubes with parameters} $\delta\in (0,1)$ and
$0<c_1\leq C_1<\infty$ if it has the following properties:
\begin{equation}\label{eq:cover}
    X
    = \bigcup_{\alpha\in \mathscr{A}_k} Q^k_{\alpha}
    \quad\text{(disjoint union) for all}~k\in\Z;
\end{equation}
\begin{equation}\label{eq:nested}
    \text{if }\ell\geq k\text{, then either }
        Q^{\ell}_{\beta}\subseteq Q^k_{\alpha}\text{ or }
        Q^k_{\alpha}\cap Q^{\ell}_{\beta}=\emptyset;
\end{equation}
\begin{equation}\label{eq:dyadicparent}
    \text{for each }(k,\alpha)\text{ and each } \ell\leq k,
    \text{ there exists a unique } \beta
    \text{ such that }Q^{k}_{\alpha}\subseteq Q^\ell_{\beta};
\end{equation}
\begin{equation}\label{eq:children}
\begin{split}
    & \text{for each $(k,\alpha)$ there exists at most $M$
        (a fixed geometric constant)  $\beta$ such that }  \\
    & Q^{k+1}_{\beta}\subseteq Q^k_{\alpha}, \text{ and }
        Q^k_\alpha =\bigcup_{\substack{Q\in\mathscr{D}_{k+1}\\
    Q\subseteq Q^k_{\alpha}}}Q;
\end{split}
\end{equation}
\begin{equation}\label{eq:contain}
    B(x^k_{\alpha},c_1\delta^k)
    \subseteq Q^k_{\alpha}\subseteq B(x^k_{\alpha},C_1\delta^k)
    =: B(Q^k_{\alpha});
\end{equation}
\begin{equation}\label{eq:monotone}
   \text{if }\ell\geq k\text{ and }
   Q^{\ell}_{\beta}\subseteq Q^k_{\alpha}\text{, then }
   B(Q^{\ell}_{\beta})\subseteq B(Q^k_{\alpha}).
\end{equation}
The set $Q^k_\alpha$ is called a \textit{dyadic cube of
generation} $k$ with center point $x^k_\alpha\in Q^k_\alpha$
and side length~$\delta^k$. The interior and closure of
$Q^k_\alpha$ are denoted by $\widetilde{Q}^k_{\alpha}$ and
$\bar{Q}^k_{\alpha}$, respectively.

\subsection{Adjacent Systems of Dyadic Cubes}

In a geometrically doubling quasi-metric space $(X,d)$, a
finite collection $\{\mathscr{D}^t\colon t=1,2,\ldots ,T\}$ of
families $\mathscr{D}^t$ is called a \textit{collection of
adjacent systems of dyadic cubes with parameters} $\delta\in
(0,1), 0<c_1\leq C_1<\infty$ and $1\leq C<\infty$ if it has the
following properties: individually, each $\mathscr{D}^t$ is a
system of dyadic cubes with parameters $\delta\in (0,1)$ and $0
< c_1 \leq C_1 < \infty$; collectively, for each ball
$B(x,r)\subseteq X$ with $\delta^{k+3}<r\leq\delta^{k+2},
k\in\Z$, there exist $t \in \{1, 2, \ldots, T\}$ and
$Q\in\mathscr{D}^t$ of generation $k$ and with center point
$^tx^k_\alpha$ such that $\rho(x,{}^tx_\alpha^k) <
2A_0\delta^{k}$ and
\begin{equation}\label{eq:ball;included}
    B(x,r)\subseteq Q\subseteq B(x,Cr).
\end{equation}

We recall from \cite{HK} the following construction.

\begin{thm}\label{thm:existence2}
    Let $(X,d)$ be a geometrically doubling quasi-metric space.
    Then there exists a collection $\{\mathscr{D}^t\colon
    t = 1,2,\ldots ,T\}$ of adjacent systems of dyadic cubes with
    parameters $\delta\in (0, (96A_0^6)^{-1}), c_1 = (12A_0^4)^{-1},
    C_1 = 4A_0^2$ and $C = 8A_0^3\delta^{-3}$. The center points
    $^tx^k_\alpha$ of the cubes $Q\in\mathscr{D}^t_k$ have, for each
    $t\in\{1,2,\ldots,T\}$, the two properties
    \begin{equation*}
        \rho(^tx_{\alpha}^k, {}^tx_{\beta}^k)
        \geq (4A_0^2)^{-1}\delta^k\quad(\alpha\neq\beta),\qquad
        \min_{\alpha}\rho(x,{}^tx^k_{\alpha})
        < 2A_0\delta^k\quad \text{for all}~x\in X.
    \end{equation*}
  \end{thm}

We recall from \cite[Remark 2.8]{KLPW}
that the number $T$ of the adjacent systems of dyadic
    cubes as in the theorem above satisfies the estimate
    \begin{equation}\label{eq:upperbound}
        T
        = T(A_0,A_1,\delta)
        \leq A_1^6(A_0^4/\delta)^{\log_2A_1}.
    \end{equation}

Further, we also recall the following result on the smallness of the
boundary.
\begin{prop}
    Suppose that $144A_0^8\delta\leq 1$. Let $\mu$ be a
    positive $\sigma$-finite measure on $X$. Then the
    collection $\{\mathscr{D}^t\colon t=1,2,\ldots ,T\}$ may be
    chosen to have the additional property that
    \[
        \mu(\partial Q) = 0
        \quad  \textup{ for all } \; Q\in\bigcup_{t=1}^{T}\mathscr{D}^t.
    \]
\end{prop}

\subsection{Muckenhoupt $A_p$ Weights}

\begin{defn}
  \label{def:Ap}
  Let $\omega(x)$ be a nonnegative locally integrable function
  on~$X$. For $1 < p < \infty$, we
  say $\omega$ is an $A_p$ \emph{weight}, written $\omega\in
  A_p$, if
  \[
    [w]_{A_p}
    := \sup_B \left(\intav_B w\right)
    \left(\intav_B
      \left(\dfrac{1}{w}\right)^{1/(p-1)}\right)^{p-1}
    < \infty.
  \]
  Here the suprema are taken over all balls~$B\subset X$.
  The quantity $[w]_{A_p}$ is called the \emph{$A_p$~constant
  of~$w$}.
For $p = 1$, we say $w$ is an $A_1$ \emph{weight},
  written $w\in A_1$, if $M(w)(x)\leq w(x)$ for $\mu$-almost every $x\in X$, and let $A_\infty := \cup_{1\leq p<\infty} A_p$ and we have
    \(
    [w]_{A_\infty}
    := \sup_B \left(\intav_B w\right)
    \exp\left(\intav_B \log \left(\frac{1}{w}\right) \right)
    < \infty.
  \)
\end{defn}

Next we note that for $w\in A_p$ the measure $w(x)d\mu(x)$ is a doubling measure on $X$. To be more precise, we have
that for all $\lambda>1$ and all balls $B\subset X$,
\begin{align}\label{doubling constant of weight}
w(\lambda B)\leq \lambda^{np}[w]_{A_p}w(B),
\end{align}
where $n$ is the upper dimension of the measure $\mu$, as in \eqref{upper dimension}.

We also point out that for $w\in A_\infty$, there exists $\gamma>0$ such that for every ball $B$,
$$ \mu\Big( \Big\{x\in B: \ w(x)\geq\gamma \intav_B w\Big\} \Big)\geq {1\over 2}\mu(B). $$
And this implies that for every ball $B$ and for all $\delta\in(0,1)$,
\begin{align}\label{e-reverse holder}
\intav_B w\le C \left(\intav_B w^\delta\right)^{1/\delta};
\end{align}
see  also \cite{LOR2}.

\subsection{Weighted BMO spaces}

Next we recall the definition of the weighted BMO space on space of homogeneous type, while we point out that the Euclidean version was first introduced by Muckenhoupt and Wheeden \cite{MW76}.
\begin{defn}\label{d-bmo}
Suppose $w \in A_\infty$.
A function $b\in L^1_{\rm loc}(X)$ belongs to
the weighted BMO space $BMO_w(X)$ if
\begin{equation*}
\|b\|_{BMO_w(X)}:=\sup_{B}{1\over w(B)}\dint_{B}
\lf|b(x)-b_{B}\r|\, d\mu(x)<\fz,
\end{equation*}
where the sup is taken over all quasi-metric balls $B\subset X$ and
$$ b_B= {1\over\mu(B)} \int_B f(y)d\mu(y). $$
\end{defn}

\subsection{Sparse Operators on Spaces of Homogeneous Type}\label{s4}

Let $\mathcal D$ be a system of dyadic cubes on $X$ as in Section 2.1. We recall the sparse operators on spaces of homogeneous type as studied in \cite{Moen, DGKLWY}.


\begin{defn} \label{D:Carleson}
Given $0<\eta<1$, a collection $\mathcal S \subset \mathcal D$ of dyadic cubes is said to be $\eta$-sparse if for
every cube $Q\in\mathcal D$,
$$ \sum_{P\in\mathcal S, P\subset Q}\mu(P)\leq {1\over \eta} \mu(Q). $$
\end{defn}

Note, that for a collection $\mathcal S \subset \mathcal D$ of dyadic cubes with the property that for $0<\eta<1$ and for every $Q\in\mathcal S$, there is a measurable subset $E_Q \subset Q$ such that $\mu(E_Q) \geq \eta \mu(Q)$ and the sets $\{E_Q\}_{Q\in\mathcal S}$ have only finite overlap, we will have that $S$ is $\eta$-sparse according to Definition \ref{D:Carleson} (following from the standard computation).

We now recall the well-known definition for sparse operator.
\begin{defn} \label{D:Sparse Operator}
Given $0<\eta<1$ and an $\eta$-sparse family $\mathcal S \subset \mathcal D$ of dyadic cubes. The
sparse operators $\mathcal A_{\mathcal S}$ is defined by
$$ \mathcal A_{\mathcal S}f(x):= \sum_{Q\in \mathcal S} f_Q \chi_Q(x).$$
\end{defn}
Following the proof of \cite[Theorem 3.1]{Moen}, we obtain that
\begin{align*}
\|\mathcal A_{\mathcal S}f\|_{L^p_w(X)}\leq C_{\eta,n,p}[w]_{A_p}^{\max\{1,{1\over p-1}\}}, \quad 1<p<\infty.
\end{align*}
Denote by $\Omega(b,B)$ the standard mean oscillation
$$ \Omega(b,B)={1\over\mu(B)} \int_B|b(x)-b_B|d\mu(x).  $$
We recall from \cite[Lemma 3.5]{DGKLWY} the following result.
\begin{lem}\label{lem b-bQ}
Let $\mathcal D$ be a dyadic system in $X$ and let $\mathcal S\subset \mathcal D$ be a $\gamma$-sparse
family. Assume that $b\in L^1_{loc}(X)$. Then there exists a ${\gamma\over 2(\gamma+1)}$-sparse family
$\tilde{\mathcal S}\subset \mathcal D$ such that $\mathcal S\subset \tilde{\mathcal S}$ and for every cube $Q\in \tilde{\mathcal S}$,
\begin{align}\label{b-bQ eee0}
|b(x)-b_Q|\leq C\sum_{R\in \tilde{\mathcal S}, R\subset Q} \Omega(b,R)\chi_R(x)
\end{align}
for a.e. $x\in Q$.
\end{lem}

\section{Upper Bound of the Maximal Commutator $C_b$}
In this section we provide the proof of Theorem \ref{thm main2}, which implies Theorem \ref{thm main1}

\subsection{Sparse domination of the maximal commutator $C_b$}
%
%

%
Given a ball $B_0\subset X$, for $x\in B_0$ we define a local grand maximal truncated operator $\mathcal M_{ B_0}$ as follows:
\begin{align*}
\mathcal M_{ B_0}f(x) := \sup_{B\ni x,\,  B\subset B_0} \esssup_{\xi\in B}  \mathcal M\big(f\chi_{4A_0B_0\backslash 4A_0B}\big)(\xi)  .
\end{align*}

Using the idea of   \cite[Lemma 3.2]{Ler1}, we can obtain the following lemma.
\begin{lem} \label{lem TM}
For $\mu$-almost every $x\in B_0$,
$$   \mathcal M(f\chi_{4A_0B_0})(x)\leq C\|\mathcal M\|_{L^1\to L^{1,\infty}} |f(x)| + \mathcal M_{ B_0}f(x).   $$

\end{lem}
\begin{proof}
Suppose that $x\in B_0$, and let $x$ be a point of approximate continuity of $ \mathcal M(f\chi_{4A_0B_0})$ (see, e.g., \cite{EG}{, p. 46 }).
Then for every $\varepsilon>0$, the sets
$$E_s(x)=\{y\in B(x,s): |\mathcal M(f\chi_{4A_0B_0})(y)-\mathcal M(f\chi_{4A_0B_0})(x)|<\varepsilon\}$$
satisfy $\lim\limits_{s\to 0}{\mu(E_s(x))\over \mu(B(x,s))}=1.$

Then for a.e. $y\in E_s(x)$,

$$ \mathcal M(f\chi_{4A_0B_0})(x)\leq  \mathcal M(f\chi_{4A_0B_0})(y)+\varepsilon\leq  \mathcal M(f\chi_{4A_0B(x,s)})(y)+\mathcal M_{ B_0}f(x)+\varepsilon.$$

Therefore, applying the weak type $(1, 1)$ of $\mathcal M$ yields

\begin{align*}
 \mathcal M(f\chi_{4A_0B_0})(x)&\leq  \essinf_{y\in E_s(x)}\mathcal M(f\chi_{4A_0B(x,s)})(y)+\mathcal M_{ B_0}f(x)+\varepsilon \\
  &\leq C\|\mathcal M\|_{L^1\to L^{1,\infty}}{1\over \mu(E_s(x))}\int_{4A_0B(x,s)}|f(x)|d\mu(x)+\mathcal M_{ B_0}f(x)+\varepsilon.
\end{align*}
Assuming additionally that $x$ is a Lebesgue point of $f$ and letting
subsequently $s\to 0$ and $\varepsilon\to 0$, we completes the proof of this lemma.
\end{proof}

Next, following from \cite{LOR} (see also \cite{DGKLWY}) we give
the sparse operator $\mathcal T_{\mathcal S,b}$ defined by
$$   \mathcal T_{\mathcal S,b}(f)(x) =\sum_{Q\in \mathcal S} |b(x)-b_Q|f_Q \chi_Q(x).     $$
And we let $\mathcal T_{\mathcal S,b}^*$ denote the adjoint operator to $\mathcal T_{\mathcal S,b}$:
$$   \mathcal T_{\mathcal S,b}^*(f)(x) =\sum_{Q\in \mathcal S} \bigg( {1\over\mu(Q)} \int_Q |b(y)-b_Q|f(y)d\mu(y)\bigg) \chi_Q(x).     $$

Then we have   the following result.
\begin{thm}\label{thm sparse1}
For every compactly supported $f\in L^\infty(X)$,
there exists $T$ dyadic systems $\mathcal D^t, t=1,2,\ldots,T$ and $\eta$-sparse families $\mathcal S_t \subset \mathcal D^t$ such that
for a.e. $x\in X$,
\begin{align}
|C_b(f)(x)|\leq C \sum_{t=1}^T \Big(\mathcal T_{\mathcal S_t,b}(|f|)(x) + \mathcal T_{\mathcal S_t,b}^*(|f|)(x) \Big).
\end{align}
\end{thm}
\begin{proof}
We recall from Section 2.2, for each ball
$B(x,r)\subseteq X$ with $\delta^{k+3}<r\leq\delta^{k+2},
k\in\Z$, there exist $t \in \{1, 2, \ldots, T\}$ and
$Q\in\mathscr{D}^t$ of generation $k$ and with center point
$^tx^k_\alpha$ such that $\rho(x,{}^tx_\alpha^k) <
2A_0\delta^{k}$ and
$    B(x,r)\subseteq Q\subseteq B(x,Cr)$. Here and in what follows, $A_0$ denotes the constant in \eqref{eqn:quasitriangleineq}.

Fix a ball $B_0\subset X$, then it is clear that there exist a positive constant $C$, $t_0 \in \{1, 2, \ldots, T\}$ and $Q_0\in\mathscr{D}^{t_0}$ such that
$    4A_0B_0\subseteq Q_0\subseteq C(4A_0B_0)$.
We now show that there exists a 
${1\over 8C_{A_0,\mu}}$-sparse
family $\mathcal F^{t_0}\subset \mathcal D^{t_0}(B_0)$ such that for a.e.
$x\in B_0$,
\begin{align}\label{ee a}
|C_b(f\chi_{4A_0B_0})(x)|\leq C\sum_{Q\in \mathcal F^{t_0}} \Big(|b(x)-b_{R_Q}||f|_{4A_0Q} + |(b-b_{R_Q})f|_{4A_0Q} \Big)\chi_Q(x).
\end{align}
Here, $R_Q$ is the dyadic cube in $\mathscr{D}^t$ for some $t \in \{1, 2, \ldots, T\}$ such that
$4A_0Q\subset R_Q \subset C(4A_0Q)$.

It suffices to prove the following recursive claim: there exist pairwise disjoint cubes $P_j\in\mathcal D^{t_0}(B_0)$ such that
$\sum_j \mu(P_j)\leq {1\over 2}\mu(B_0)$ and
\begin{align}\label{ee a claim}
|C_b(f\chi_{4A_0B_0})(x)|\chi_{B_0}&\leq C  \Big(|b(x)-b_{Q_0}||f|_{4A_0B_0} + |(b-b_{Q_0})f|_{4A_0B_0} \Big)\\
&\quad+ \sum_j |C_b(f\chi_{4A_0P_j})(x)|\chi_{P_j}\nonumber
\end{align}
a.e. on $B_0$. 
Here we have a ${1\over 2}$-sparse family since the sets $E_{Q}=Q\backslash \cup_jP_j$, and then we can appeal to the discussion after Definition \ref{D:Carleson}.


Now observe that for arbitrary pairwise disjoint cubes $P_j\in\mathcal D^{t_0}(B_0)$,
\begin{align*}
&|C_b(f\chi_{4A_0B_0})(x)|\chi_{B_0}\\
&=|C_b(f\chi_{4A_0B_0})(x)|\chi_{B_0\backslash \cup_jP_j} + \sum_j |C_b(f\chi_{4A_0B_0})(x)|\chi_{P_j} \\
&\leq |C_b(f\chi_{4A_0B_0})(x)|\chi_{B_0\backslash \cup_jP_j} + \sum_j |C_b(f\chi_{4A_0B_0\backslash 4A_0P_j})(x)|\chi_{P_j}+ \sum_j |C_b(f\chi_{4A_0P_j})(x)|\chi_{P_j}.
\end{align*}
Hence, in order to prove the recursive claim \eqref{ee a claim}, it suffices to show that
one can select pairwise disjoint cubes $P_j\in\mathcal D^{t_0}(B_0)$ with
$\sum_j\mu(P_j)\leq {1\over2}\mu(B_0)$ and such that for a.e. $x\in B_0$,
\begin{align}\label{ee b}
 &|C_b(f\chi_{4A_0B_0})(x)|\chi_{B_0\backslash \cup_jP_j} + \sum_j |C_b(f\chi_{4A_0B_0\backslash 4A_0P_j})(x)|\chi_{P_j}\\
 &\leq C\Big(  |b(x)-b_{Q_0}| |f|_{4A_0B_0}  +  |(b-b_{Q_0})f|_{4A_0B_0}  \Big).\nonumber
\end{align}

To see this, by definition, we obtain that
\begin{align}\label{ee b-c}
 &|C_b(f\chi_{4A_0B_0})(x)|\chi_{B_0\backslash \cup_jP_j}(x) + \sum_j |C_b(f\chi_{4A_0B_0\backslash 4A_0P_j})(x)| \chi_{P_j}(x)\\
 &\leq \sup_{ \bar B \ni x} {1\over \mu(\bar B)} \int_{\bar B} |b(x)-b(y)| |f(y)|\chi_{4A_0B_0}(y) d\mu(y)  \chi_{B_0\backslash \cup_jP_j}(x) \nonumber\\
 &+  \sum_j  \sup_{ \bar B \ni x} {1\over \mu(\bar B)} \int_{\bar B} |b(x)-b(y)| |f(y)|\chi_{4A_0B_0\backslash 4A_0P_j}(y) d\mu(y)  \chi_{P_j}(x)    \nonumber\\
 &\leq |b(x)-b_{Q_0}| \sup_{ \bar B \ni x} {1\over \mu(\bar B)} \int_{\bar B}  |f(y)|\chi_{4A_0B_0}(y) d\mu(y)  \chi_{B_0\backslash \cup_jP_j}(x) \nonumber\\
 &\quad+ |b(x)-b_{Q_0}| \sum_j  \sup_{ \bar B \ni x} {1\over \mu(\bar B)} \int_{\bar B}  |f(y)|\chi_{4A_0B_0\backslash 4A_0P_j}(y) d\mu(y)  \chi_{P_j}(x)    \nonumber\\
 &\quad+\sup_{ \bar B \ni x} {1\over \mu(\bar B)} \int_{\bar B} |b(y)-b_{Q_0}| |f(y)|\chi_{4A_0B_0}(y) d\mu(y)  \chi_{B_0\backslash \cup_jP_j}(x) \nonumber\\
 &\quad+  \sum_j  \sup_{ \bar B \ni x} {1\over \mu(\bar B)} \int_{\bar B} |b(y)-b_{Q_0}| |f(y)|\chi_{4A_0B_0\backslash 4A_0P_j}(y) d\mu(y)  \chi_{P_j}(x).    \nonumber\\
 \nonumber
 \end{align}

We now choose $\alpha$ such that the set $E:=E_1\cup E_2 \cup E_3\cup E_4$, with
\begin{align*}
E_1=\{x\in B_0: |f(x)|> \alpha |f|_{4A_0B_0}\},
\end{align*}
\begin{align*}
E_2=  \Big\{ x\in B_0:  \mathcal M_{ B_0}f(x)  > \alpha C |f|_{4A_0B_0} \Big\},
\end{align*}
\begin{align*}
E_3=\{x\in B_0: |(b(x)-b_{Q_0})f(x)|> \alpha |(b-b_{Q_0})f|_{4A_0B_0}\},
\end{align*}
and
\begin{align*}
E_4=
 \{ x\in B_0: \mathcal M_{ B_0}\big((b -b_{Q_0})f\big)(x)  > \alpha   C  |(b-b_{Q_0})f|_{4A_0B_0} \},
\end{align*}
will satisfy $$\mu(E)\leq {1\over 2^{2+n}} \mu(B_0).$$

We now apply the Calder\'on--Zygmund decomposition to the function $\chi_E$ on $B_0$ at the height
$\lambda= {1\over 2^{n+1}}$, where $n$ is the upper dimension of the measure $\mu$ as in \eqref{upper dimension}, to obtain the pairwise disjoint cubes $P_j\in \mathcal D^{t_0}(B_0)$ such that
$$ \chi_E(x)\leq {1\over 2^{n+1}} \quad {\rm \ a.e.\ } x\not\in \cup_jP_j $$
and hence we have that $ \mu(E\backslash \cup_j P_j)=0$.
Moreover, we have that
$$ \sum_j\mu(P_j)  = \mu\Big( \bigcup_j P_j\Big)\leq 2^{n+1}\mu(E)\leq {1\over2}\mu(B_0),$$
and that
$$ {1\over 2^{n+1}}\leq {1\over \mu(P_j)} \int_{P_j}\chi_E(x)d\mu(x) = {\mu(P_j\cap E)\over \mu(P_j)}\leq {1\over2}, $$
which implies that
$$
 P_j\cap E^c\not=\emptyset.$$

\bigskip
Therefore, we observe that for each $P_j$, since $P_j\cap E^c\not=\emptyset$, we have that
$$\mathcal{M}_{ B_0}\big((b -b_{Q_0})f\big)(x) \leq \alpha C |(b-b_{Q_0})f|_{4A_0B_0} $$
for some $x\in P_j$, which implies that
 $$    \esssup_{\xi\in P_j} \sup_{ \bar B \ni \xi} {1\over \mu(\bar B)} \int_{\bar B} |b(y)-b_{Q_0}| |f(y)|\chi_{4A_0B_0\backslash 4A_0P_j}(y) d\mu(y) \leq \alpha C |(b-b_{Q_0})f|_{4A_0B_0}.   $$
 Similarly,
 we have
 $$   \esssup_{\xi\in P_j} \sup_{ \bar B \ni \xi} {1\over \mu(\bar B)} \int_{\bar B}  |f(y)|\chi_{4A_0B_0\backslash 4A_0P_j}(y) d\mu(y) \leq \alpha C |f|_{4A_0B_0} .  $$

Also, by   Lemma \ref{lem TM}, we have that
\begin{align*}
\sup_{ \bar B \ni x} {1\over \mu(\bar B)} \int_{\bar B}  |f(y)|\chi_{4A_0B_0}(y) d\mu(y)\leq C  |f(x)| + \mathcal M_{ B_0}f(x).
\end{align*}
and
\begin{align*}
\sup_{ \bar B \ni x} {1\over \mu(\bar B)} \int_{\bar B} |b(y)-b_{Q_0}| |f(y)|\chi_{4A_0B_0}(y) d\mu(y) \leq C   |b(x)-b_{Q_0}| |f(x)| + \mathcal M_{ B_0}\big((b-b_{Q_0})f\big)(x).
\end{align*}

Since $\mu(E\backslash \cup_j P_j)=0$, we have that from the definition of the set $E$, the following estimates
$$   |f(x)|\leq \alpha |f|_{4A_0B_0}, \quad   |(b(x)-b_{Q_0})f(x)|\leq \alpha |(b-b_{Q_0})f|_{4A_0B_0} $$
  hold for $\mu$-almost every $x\in B_0\backslash \cup_j P_j$, and also
$$  \mathcal M_{ B_0}f(x)  \leq \alpha C |f|_{4A_0B_0} ,\quad  \mathcal M_{ B_0}\big((b -b_{Q_0})f\big)(x) \leq \alpha C |(b-b_{Q_0})f|_{4A_0B_0} $$
  hold for $\mu$-almost every $x\in B_0\backslash \cup_j P_j$.

Combining these fact with \eqref{ee b-c}, we see that \eqref{ee b} holds, which further implies that
\eqref{ee a} holds.

We now consider the partition of the space as follows. Suppose $f$ is supported in a ball $B_0\subset X$.  We have
$$ X = \bigcup_{j=0}^{\infty} 2^jB_0. $$

First, we note that  the ball $B_0$ is covered by $4A_0B_0$.   Consider the annuli $U_j:= 2^{j}B_0\backslash 2^{j-1}B_0$ for $j\geq1$. It is clear that we can choose the balls $\{\tilde B_{j,\ell}\}_{\ell=1}^{L_j}$ with radius $2^{j-2}r_B$ to cover $U_j$, satisfying that the center of each the ball $\tilde B_{j,\ell}$ is in  $U_j\not=\emptyset$ and that
$\sup_j L_j\leq C_{A_0,\mu}$, where $C_{A_0,\mu}$ is an absolute constant depending on  $A_0$ and $C_\mu$ only, here $C_\mu$ is the constant as in  \eqref{doubling condition}. Moreover, we also have that
for each such $\tilde B_{j,\ell}$, the enlargement $4A_0 B_{j,\ell}$ covers $B_0$.
Also, we note that for each $\tilde B_{j,\ell}$, there exist a positive constant $C$, $t_{j,\ell} \in \{1, 2, \ldots, T\}$ and $\tilde Q_{j,\ell}\in\mathscr{D}^{t_{j,\ell}}$ such that
$    4A_0\tilde B_{j,\ell}\subseteq \tilde Q_{j,\ell}\subseteq C(4A_0\tilde B_{j,\ell})$.

We now apply \eqref{ee a} to each $\tilde B_{j,\ell}$, then
we obtain a ${1\over 2}$-sparse  family $\tilde F_{j,\ell} \subset \mathscr{D}^{t_{j,\ell}}(\tilde B_{j,\ell})$ such that
\eqref{ee a} holds for a.e. $x\in \tilde B_{j,\ell}$.  

Now we set $\mathcal F = \cup_{j,\ell} \tilde F_{j,\ell}$.
Note that the balls $\tilde B_{j,\ell}$ are overlapping at most $4C_{A_0,\mu}$ times.   Then we obtain that
$\mathcal F $ is a ${1\over 8C_{A_0,\mu}} $-sparse family and for a.e. $x\in X$,
 \begin{align}\label{ee a a}
|C_b(f)(x)|\leq C\sum_{Q\in \mathcal F} \Big(|b(x)-b_{R_Q}||f|_{4A_0Q} + |(b-b_{R_Q})f|_{4A_0Q} \Big)\chi_Q(x).
\end{align}

Since $4A_0Q\subset R_Q$, and it is clear that $\mu(R_Q)\leq \overline C\mu(4A_0Q)$, we obtain that
$|f|_{4A_0Q} \leq C |f|_{R_Q} $. Next, we further set
$$ \mathcal S_{t}=\{R_Q\in \mathcal D^t:\ Q\in\mathcal F\}, \quad t\in\{1,2,\ldots,T\}, $$
 and from the fact that $\mathcal F$ is ${1\over 8C_{A_0,\mu}}$-sparse, we can obtain that
 each family $\mathcal S_{t}$ is ${1\over 8C_{A_0,\mu}\overline C}$-sparse.
 Now we let
 $$\eta={1\over 8C_{A_0,\mu}\overline C}. $$
 Then it follows that
 \begin{align}\label{ee a a final}
|C_b(f)(x)|\leq C\sum_{t=1}^T\sum_{R\in \mathcal S_t} \Big(|b(x)-b_{R}||f|_{R} + |(b-b_{R})f|_{R} \Big)\chi_R(x),
\end{align}
finishing the proof.
\end{proof}

\subsection{Proof of Theorem \ref{thm main2}}

\begin{proof}
Let $\mathcal D$ be a dyadic system in $(X,d,\mu)$ and let
$\mathcal S$ be a sparse family from $\mathcal D$. By Theorem \ref{thm sparse1}, we only need to prove
$$\| \mathcal T_{\mathcal S,b} : L^p_{\lambda_1}(X) \rightarrow L^p_{\lambda_2}(X) \|  \le C\Big([\lambda_1]_{A_p} [\lambda_2]_{A_p}  \Big)^{\max\{1,{1\over p-1}\}}  \| b \| _{{\rm BMO}_{\nu}(X)} $$
and
$$\|\mathcal T_{\mathcal S,b}^* : L^p_{\lambda_1}(X) \rightarrow L^p_{\lambda_2}(X) \|  \le C\Big([\lambda_1]_{A_p} [\lambda_2]_{A_p}  \Big)^{\max\{1,{1\over p-1}\}}  \| b \| _{{\rm BMO}_{\nu}(X)} .$$

By duality, we have that
\begin{align}\label{A Lpnorm}
\|\mathcal T_{\mathcal S,b}f\|_{L^p_{\lambda_2}(X)} &\leq \sup_{g: \|g\|_{L^{p'}_{\lambda_2}(X)}=1}\sum_{Q\in\mathcal S}  \bigg(\int_Q |g(x){\lambda_2}(x)||b(x)-b_Q| d\mu(x)\bigg)|f |_Q .
\end{align}

Now by Lemma \ref{lem b-bQ}, there exists a sparse family
$\tilde{\mathcal S}\subset \mathcal D$ such that $\mathcal S\subset \tilde{\mathcal S}$ and for every cube $Q\in \tilde{\mathcal S}$, for $\mu$-almost every $x\in Q$,
\begin{align*}
|b(x)-b_Q|\leq C\sum_{P\in \tilde{\mathcal S}, P\subset Q} \Omega(b,P)\chi_P(x).
\end{align*}

Since $b$ is in ${\rm BMO}_{\nu }(X)$, then we have for $\mu$-almost every $x\in Q$
\begin{align*}
|b(x)-b_Q|\leq C\|b\|_{{\rm BMO}_{\nu }(X)} \sum_{P\in \tilde{\mathcal S}, P\subset Q}  {\nu (P)\over \mu(P)} \chi_P(x).
\end{align*}
Then combining this estimate and inequality \eqref{A Lpnorm}, we further have
\begin{align*}
&\|\mathcal T_{\mathcal S,b}f\|_{L^p_{\lambda_2}(X)} \\
&\leq C \|b\|_{{\rm BMO}_{\nu }(X)} \sup_{g: \|g\|_{L^{p'}_{\lambda_2}(X)}=1}\sum_{Q\in\mathcal S}  \Bigg( \int_Q |g(x){\lambda_2}(x)|   \bigg(\sum_{P\in \tilde{\mathcal S}, P\subset Q}  {\nu (P)\over \mu(P)} \chi_P(x)\bigg)    d\mu(x)\Bigg)|f |_Q \\
&\leq C \|b\|_{{\rm BMO}_{\nu }(X)} \sup_{g: \|g\|_{L^{p'}_{\lambda_2}(X)}=1}\sum_{Q\in\mathcal S}  \Bigg(\sum_{P\in \tilde{\mathcal S}, P\subset Q}  {\nu (P)\over \mu(P)} \int_Q |g(x){\lambda_2}(x)|     \chi_P(x)     d\mu(x)\Bigg)|f |_Q   \\
&\leq C \|b\|_{{\rm BMO}_{\nu }(X)} \sup_{g: \|g\|_{L^{p'}_{\lambda_2}(X)}=1}\sum_{Q\in\mathcal S}  \Bigg(\sum_{P\in \tilde{\mathcal S}, P\subset Q}     \nu (P) |g {\lambda_2} |_P     \Bigg)|f |_Q   \\
&\leq C \|b\|_{{\rm BMO}_{\nu }(X)} \sup_{g: \|g\|_{L^{p'}_{\lambda_2}(X)}=1}\sum_{Q\in\mathcal S}  \Bigg(\int_Q\sum_{P\in \tilde{\mathcal S}, P\subset Q}      |g {\lambda_2} |_P\chi_P(x)\nu(x)d\mu(x)     \Bigg)|f |_Q   \\
&\leq C \|b\|_{{\rm BMO}_{\nu }(X)} \sup_{g: \|g\|_{L^{p'}_{\lambda_2}(X)}=1}\sum_{Q\in\mathcal S}  \Bigg(\int_Q\mathcal A_{\tilde{\mathcal S}}(|g {\lambda_2}|)(x)\nu(x)d\mu(x)     \Bigg)|f |_Q   \\
&\leq C \|b\|_{{\rm BMO}_{\nu }(X)} \sup_{g: \|g\|_{L^{p'}_{\lambda_2}(X)}=1}\int_X\sum_{Q\in\mathcal S} |f |_Q\chi_Q(x)  \mathcal A_{\tilde{\mathcal S}}(|g {\lambda_2}|)(x)\nu(x)d\mu(x)       \\
&\leq C \|b\|_{{\rm BMO}_{\nu }(X)} \sup_{g: \|g\|_{L^{p'}_{\lambda_2}(X)}=1}\int_X\mathcal A_{\tilde{\mathcal S}}(|f|)(x)  \mathcal A_{\tilde{\mathcal S}}(|g {\lambda_2}|)(x)\nu(x)d\mu(x)     .
\end{align*}

Observe that $\mathcal A_{\tilde{\mathcal S}}$ is self-adjoint.
Then by H\"older's inequality, we have

\begin{align*}
\|\mathcal T_{\mathcal S,b}f\|_{L^p_{\lambda_2}(X)}
&\leq C \|b\|_{{\rm BMO}_{\nu }(X)} \sup_{g: \|g\|_{L^{p'}_{\lambda_2}(X)}=1}\int_X \mathcal A_{\tilde{\mathcal S}}\big(\mathcal A_{\tilde{\mathcal S}}(|f|)\nu \big)(x)  |g(x)| {\lambda_2}(x)d\mu(x) \\
&\leq C \|b\|_{{\rm BMO}_{\nu }(X)} \sup_{g: \|g\|_{L^{p'}_{\lambda_2}(X)}=1}\|\mathcal A_{\tilde{\mathcal S}}\big(\mathcal A_{\tilde{\mathcal S}}(|f|)\nu \big)\|_{L^p_{\lambda_2}(X)}\|g\|_{L^{p'}_{\lambda_2}(X)} \\
&\leq C \|b\|_{{\rm BMO}_{\nu }(X)} [\lambda_2]_{A_p}^{\max\{1,{1\over p-1}\}}  \| \mathcal A_{\tilde{\mathcal S}}(|f|)\nu  \|_{L^p_{\lambda_2}(X)}  \\
&\leq C \|b\|_{{\rm BMO}_{\nu }(X)} [\lambda_2]_{A_p}^{\max\{1,{1\over p-1}\}}  \| \mathcal A_{\tilde{\mathcal S}}(|f|)   \|_{L^p_{\lambda_1}(X)}  \\
&\leq C \|b\|_{{\rm BMO}_{\nu }(X)}\Big([\lambda_1]_{A_p} [\lambda_2]_{A_p}  \Big)^{\max\{1,{1\over p-1}\}}  \| f    \|_{L^p_{\lambda_1}(X)}.
\end{align*}

 Similarly,  we can obtain

\begin{align*}
\|\mathcal T_{\mathcal S,b}^*f\|_{L^p_{\lambda_2}(X)}
&\leq C \|b\|_{{\rm BMO}_{\nu }(X)}\Big([\lambda_1]_{A_p} [\lambda_2]_{A_p}  \Big)^{\max\{1,{1\over p-1}\}}  \| f    \|_{L^p_{\lambda_1}(X)}.
\end{align*}
This completes the proof of  Theorem \ref{thm main2}.
\end{proof}

We point out that the quantitative estimate with respect to the weights $\lambda_1$ and $\lambda_2$ that we obtain here
is better  comparing to the methods used in \cite{GHST} and \cite{AGKM}.

\begin{itemize}

\item In \cite{GHST}, for $\lambda_{1}^{-1},\lambda_{2}^{-1}\in A_1$, $\nu= \lambda_1\lambda_2^{-{1 }}$ and $b\in {\rm BMO}_\nu$, they obtained that
$$\|\lambda_{2}C_b(f)\|_{L^{\infty}}\leq C'(\lambda_1,\lambda_2)\|b\|_{{\rm BMO}_\nu}\|\lambda_1 f\|_{L^{\infty}}.$$
Then by extrapolation they obtained that for $1<p<\infty$, $\lambda_1,\lambda_2\in A_p$, $\nu = \lambda_1^{1\over p}\lambda_2^{-{1\over p}}$ and $b\in {\rm BMO}_\nu$,
$$\| C_b(f)\|_{L^{p}_{\lambda_{2}}}\leq C(b,\lambda_1,\lambda_2) \| f\|_{L^{p}_{\lambda_1}}.$$
They only showed that $C(b,\lambda_1,\lambda_2)$ depends on $b,\lambda_1,\lambda_2$.

\item In \cite{AGKM}, to prove the upper bound of $C_b$ in $\mathbb R^n$, they first proved that (see Corollary 1.11 in \cite{AGKM})
$$C_b(f)(x)\leq C \|b\|_* \mathcal{M}^2(f)(x),$$
where $ \|b\|_*$ is the norm
$$  \|b\|_*=\Big(\sup_{B} {1\over |B|} \int_B |b(x)-b_B|^pdx\Big)^{1\over p}.$$
By using John--Nirenberg's inequality we know that it is equivalent to the BMO norm, that is
$$ \|b\|_{{\rm BMO}(\mathbb R^n)}\leq \|b\|_*\leq C  \|b\|_{{\rm BMO}(\mathbb R^n)}.$$

Using this result directly, one can only get the one weight boundedness for $C_b$. That is, for $1<p<\infty$, $\lambda\in A_p$ and $b\in {\rm BMO}(\R^n)$,
$$\| C_b(f)\|_{L^{p}_{\lambda }}\leq C\|b\|_{{\rm BMO}(\R^n) }[\lambda]_{A_p}^{2\max\{1,{1\over p-1}\}} \| f\|_{L^{p}_{\lambda }}.$$

If we try to use this approach to obtain the two weight upper bound for $C_b$ in our setting, we will first need to obtain the quantitative estimate for the John--Nirenberg inequality for weighted BMO. However, this quantitative estimate, together with the quantitative estimate for $\mathcal{M}^2(f)(x)$, is certainly larger than what we obtained by using sparse operator.

For completeness, we provide the quantitative estimate for the John--Nirenberg inequality for weighted BMO as follows.

\end{itemize}

We now provide the quantitative version of Bloom's estimate \cite[Lemma 4.3]{B} on spaces of homogeneous type.

\begin{lem}\label{lem b-bQ}
Suppose $1<p<\infty$, $\lambda_1,\lambda_2\in A_p$, $\nu:= \lambda_1^{1\over p}\lambda_2^{-{1\over p}}$. Suppose $b\in {\rm BMO}_{\nu}(X)$. Then
there exists $\varepsilon>0$ such that for all $1\leq r\leq p'+\varepsilon$,
\begin{align}\label{b-bQ eee0}
&{1\over \mu(B)}\int_{B}|b(x)-b_B|^r\lambda_{1}^{-{r\over p}}(x)d\mu(x)\nonumber\\
& \lesssim
                \begin{cases}
                  \|b\|_{{\rm BMO}_{\nu}}^{r}[\lambda_{2}]_{A_p}^{\frac{r}{p}} \bigg({\displaystyle1\over\displaystyle \mu(B)}\displaystyle\int_{B}\lambda_{2}^{-{1\over p}}(x)d\mu(x)\bigg)^r,\quad 1\leq r\leq p'\\[18pt]
                  [\lambda_{1}]_{A_p}^{\frac{r}{p}} [\lambda_{2}]_{A_p}^{\frac{r}{p}} \|b\|_{{\rm BMO}_{\nu}}^{r} \left(\frac{\displaystyle1}{\displaystyle\mu(B)}\displaystyle\int_{B}\lambda_{2}^{\frac{-1}{p}}(x)d\mu(x)\right)^{r},\quad p'< r\leq p'+\varepsilon\\
                \end{cases}
 \end{align}
for  every ball $B$.
\end{lem}
\begin{proof}
We will prove result when $r\leq p'$ using Holder's Inequality.
Without loss of generality we will assume $r=p'$ to show $\eqref{b-bQ eee0}$. Similar proof works for any $r<p'$. Fix a ball $B$ and choose $s= \frac{1}{p'}$ to use Holder's Inequality.
\begin{eqnarray*}
&&\frac{1}{\mu(B)}\int_{B}|b(x)-b_B|^{p'}\lambda_{1}^{\frac{-p'}{p}}d\mu(x)\\
 &\leq & \left(\frac{1}{\nu(B)}\int_{B}|b(x)-b_B|d\mu(x)\right)^{p'}\frac{\nu(B)^{p'}}{\mu(B)}\left(\int_{B}\lambda_{1}^{\frac{-p'}{p(1-p')}}(x)d\mu(x)\right)^{1-p'}\\
&\leq&\|b\|_{{\rm BMO}_{\nu}}^{p'}\frac{\nu(B)^{p'}}{\mu(B)}\left(\int_{B}\lambda_{1}^{\frac{-p'}{p(1-p')}}d\mu(x)\right)^{1-p'}\\
&=& \|b\|_{{\rm BMO}_{\nu}}^{p'}\frac{1}{\mu(B)}\left(\int_{B}\nu(x)d\mu(x)\right)^{p'}\left(\int_{B}\lambda_{1}^{\frac{-p'}{p(1-p')}}(x)d\mu(x)\right)^{1-p'}\\
&=& \|b\|_{{\rm BMO}_{\nu}}^{p'}\frac{1}{\mu(B)}\left(\int_{B}(\nu(x)^{p'})^\frac{1}{p'}(x)d\mu(x)\right)^{p'}\left(\int_{B}(\lambda_{1}^{\frac{-p'}{p}})^\frac{1}{(1-p')}(x)d\mu(x)\right)^{1-p'}\\
&\leq& \|b\|_{{\rm BMO}_{\nu}}^{p'}\frac{1}{\mu(B)}\left(\int_{B}\lambda_{1}^{\frac{p'}{p}}(x)\lambda_{2}^{\frac{-p'}{p}}(x)\lambda_{1}^{\frac{-p'}{p}}(x)d\mu(x)\right)\hspace{1 cm}\textbf{(Reverse Holder's)}\\
&\leq& \|b\|_{{\rm BMO}_{\nu}}^{p'}\frac{1}{\mu(B)}\left(\int_{B}\lambda_{2}^{\frac{-p'}{p}}(x)d\mu(x)\right)\\
&\leq&\|b\|_{{\rm BMO}_{\nu}}^{p'}[\lambda_{2}]_{A_p}^{\frac{p'}{p}}\left(\frac{1}{\mu(B)}\int_{B}\lambda_{2}^{\frac{-1}{p}}(x)d\mu(x)\right)^{p'}.
\end{eqnarray*}
Here we used Holder's Inequality in the last step to complete the proof.

Now we will show the proof when $r>p'$. We choose an index  $r$ for which reverse Holder's Inequality holds for weights $\lambda_{1}^{-\frac{p'}{p}}$ and $\lambda_{2}^{-\frac{p'}{p}}$ with exponent $1+\delta = \frac{r}{p'}$. Fix a ball $B$ and let $x\in B$, then

\begin{equation*}
    \frac{1}{\mu(B)}\int_{B}|b(x)-b_{B}|d\mu(x) \leq \|b\|_{{\rm BMO}_{\nu}}\frac{1}{\mu(B)}\int_{B}\nu(x)d\mu(x).
\end{equation*}
Hence we have $\hat{b}\leq \|b\|_{{\rm BMO}_{\nu}}\nu^{*} $, where
\begin{eqnarray*}
\hat{b}&:=&\sup_{B}\frac{1}{\mu(B)}\int_{B}|b(x)-b_{B}|d\mu(x),\\
\nu^{*}&:=&\sup_{B}\frac{1}{\mu(B)}\int_{B}\nu(x)d\mu(x).\\
\end{eqnarray*}
Following the proof \cite[Proposition 2.1]{AD} the following Sharp function estimate holds for some constant $C_{X,b}$ depending on the underlying space $X$ and the operator $b$,
\begin{align*}
{1\over \mu(B)}\int_{B}|b(x)-b_B|^r\lambda_{1}^{-{r\over p}}(x)d\mu(x)& \leq C_{X,b} [\lambda_{1}^{\frac{-r}{p}}]_{A_p}\int_{B} (\hat{b})^{r} \lambda_{1}^{-{r\over p}}(x)d\mu(x) \\
&\leq C_{X,b} [\lambda_{1}^{\frac{-r}{p}}]_{A_p} \|b\|_{{\rm BMO}_{\nu}}^{r}\int_{B}(\nu^{*})^{r}\lambda_{1}^{-{r\over p}}(x)d\mu(x).
\end{align*}
We claim that $\lambda_{1}^{\frac{-r}{p}} \in A_r$ when $\lambda_1 \in A_p$. We will now begin to prove our claim.
\begin{eqnarray*}
&&{1\over \mu(B)}\int_{B}\lambda_{1}(x)^{\frac{-r}{p}}\left(\frac{1}{\mu(B)}\int_{B}\lambda_{1}^{\frac{r'}{p}}(x)d\mu(x)\right)^{\frac{r}{r'}}\\
&\leq&{1\over \mu(B)}\int_{B}\lambda_{1}(x)^{\frac{-r}{p}}\left(\frac{1}{\mu(B)}\int_{B}\lambda_{1}(x)d\mu(x)\right)^{\frac{r}{p}}\hspace{1 cm} \textbf{(Holder's Inequality)}\\
&\leq&\left(\frac{1}{\mu(B)}\int_{B}\lambda_{1}^{\frac{-p'}{p}}(x)d\mu(x)\right)^{\frac{r}{p'}}\left(\frac{1}{\mu(B)}\int_{B}\lambda_{1}(x)d\mu(x)\right)^{\frac{r}{p}}\hspace{1 cm}\textbf{(Reverse Holder's)}\\
&\leq&\left(\frac{1}{\mu(B)}\int_{B}(\lambda_{1}^{-1})^{\frac{1}{p-1}}(x)d\mu(x)\right)^{\frac{r(p-1)}{p}}\left(\frac{1}{\mu(B)}\int_{B}\lambda_{1}(x)d\mu(x)\right)^{\frac{r}{p}} \\
&\leq&\left( \left(\frac{1}{\mu(B)}\int_{B}(\lambda_{1}^{-1})^{\frac{1}{p-1}}(x)d\mu(x)\right)^{p-1}\frac{1}{\mu(B)}\int_{B}\lambda_{1}(x)d\mu(x)\right)^{\frac{r}{p}}\\
&\leq& [\lambda_{1}]_{A_p}^{\frac{r}{p}}.
\end{eqnarray*}
Last line above follows as we have $\lambda_{1}\in A_p$. Hence we have our claim that $\lambda_{1}^{\frac{-r}{p}} \in A_r$. Now following the proof of Muckenhoupt's Theorem as in \cite[Theorem I]{cr1} we get the following
\begin{align*}
{1\over \mu(B)}\int_{B}|b(x)-b_B|^r\lambda_{1}^{-{r\over p}}(x)d\mu(x)& \leq C_{X,b}[\lambda_{1}]_{A_p}^{\frac{r}{p}} \|b\|_{{\rm BMO}_{\nu}}^{r}\int_{B}(\nu)^{r}\lambda_{1}^{-{r\over p}}(x)d\mu(x)\\
&= C_{X,b} [\lambda_{1}]_{A_p}^{\frac{r}{p}}\|b\|_{{\rm BMO}_{\nu}}^{r}\int_{B}\lambda_{2}^{\frac{-r}{p}}d\mu(x).
\end{align*}

Similarly we can show that $\lambda_{2}^{\frac{-r}{p}}\in A_{r}$ and
\begin{equation*}
\frac{1}{\mu(B)}\int_{B}\lambda_{2}^{\frac{-r}{p}}(x)d\mu(x)\left(\frac{1}{\mu(B)}\int_{B}\lambda_{2}(x)d\mu(x)\right)^{\frac{r}{p}}\leq  [\lambda_{2}]_{A_p}^{\frac{r}{p}},
\end{equation*}
So we get
\begin{equation*}
\frac{1}{\mu(B)}\int_{B}\lambda_{2}^{\frac{-r}{p}}(x)d\mu(x)\leq [\lambda_{2}]_{A_p}^{\frac{r}{p}} \left(\frac{1}{\mu(B)}\int_{B}\lambda_{2}(x)d\mu(x)\right)^{\frac{-r}{p}}.
\end{equation*}
Using Cauchy--Schwartz' inequality, we have
\begin{equation*}
 \frac{1}{\mu(B)}\int_{B}\lambda_{2}^{\frac{1}{p}}(x)d\mu(x)\frac{1}{\mu(B)}\int_{B}\lambda_{2}^{\frac{-1}{p}}(x)d\mu(x)\geq 1.
\end{equation*}
So we have
\begin{eqnarray*}
\frac{1}{\mu(B)}\int_{B}|b(x)-b_B|^{r}\lambda_{1}^{\frac{-r}{p}}d\mu(x)&\leq& C_{X,b}[\lambda_{1}]_{A_p}^{\frac{r}{p}}\|b\|_{{\rm BMO}_{\nu}}^{r}\int_{B}\lambda_{2}^{\frac{-r}{p}}(x)d\mu(x)\\
&\leq& C_{X,b}[\lambda_{1}]_{A_p}^{\frac{r}{p}}\|b\|_{{\rm BMO}_{\nu}}^{r}[\lambda_{2}]_{A_p}^{\frac{r}{p}} \left(\frac{1}{\mu(B)}\int_{B}\lambda_{2}^{\frac{1}{p}}(x)d\mu(x)\right)^{-r}\\
&\leq& C_{X,b}[\lambda_{1}]_{A_p}^{\frac{r}{p}} [\lambda_{2}]_{A_p}^{\frac{r}{p}} \|b\|_{{\rm BMO}_{\nu}}^{r} \left(\frac{1}{\mu(B)}\int_{B}\lambda_{2}^{\frac{-1}{p}}(x)d\mu(x)\right)^{r}.
\end{eqnarray*}

The proof of Lemma \ref{lem b-bQ} is complete.
\end{proof}

\section{Lower Bound of the Commutator $C_b $ }

\begin{proof}[Proof of Theorem \ref{thm main3}]

Suppose $b\in L^1_{loc}(X)$ with $\| {C}_b : L^p_{\lambda_1}(X) \rightarrow L^p_{\lambda_2}(X) \|<\infty$.

For any fixed $B\subset X$, by H\"older's inequality, we have
\begin{align}\label{lower bd1}
{1\over \nu(B)}\int_{B}|b(x)-b_B|d\mu(x)
&\lesssim \inf_{c}{1\over \nu(B)}\int_{B}|b(x)-c|d\mu(x) \nonumber\\
&\lesssim \inf_{y\in B}{1\over \nu(B)}\int_{B}|b(x)-b(y)|d\mu(x)\nonumber\\
&\lesssim  {1\over \nu(B)}{1\over \lambda_2(B)}\int_{B} \int_B|b(x)-b(y)|d\mu(x) \lambda_2(y)d\mu(y)\nonumber\\
&\lesssim  {1\over \nu(B)}{\mu(B)\over \lambda_2(B)}\int_{B}{C}_b(\chi_B)(y)   \lambda_2(y)d\mu(y)\nonumber\\
&\lesssim  {1\over \nu(B)}{\mu(B)\over \lambda_2(B)^{1\over p}}\Big(\int_{B}| {C}_b(\chi_B)(y)|^p   \lambda_2(y)d\mu(y)\Big)^{1\over p}\nonumber\\
&\lesssim  {1\over \nu(B)}{\mu(B)\over \lambda_2(B)^{1\over p}}\| {C}_b(\chi_B)\|_{L_{\lambda_2}^{p}}\nonumber\\
&\lesssim  \|  {C}_b : L^p_{\lambda_1}(X) \rightarrow L^p_{\lambda_2}(X) \| {1\over \nu(B)}{\mu(B)\over \lambda_2(B)^{1\over p}}\lambda_1(B)^{1\over p}.
\end{align}

Using (\ref{e-reverse holder}) and H\"older's inequality, we can obtain
\begin{align*}
{1\over \mu(B)}\int_{B}\lambda_1(x)d\mu(x)
&\lesssim \Big({1\over \mu(B)}\int_{B}\lambda_1(x)^{1\over 1+p}d\mu(x)\Big)^{1+p}\\
&\lesssim \Big({1\over \mu(B)}\int_{B}\nu(x)^{p\over 1+p}\lambda_2(x)^{1\over 1+p}d\mu(x)\Big)^{1+p}\\
&\lesssim \Big({1\over \mu(B)}\int_{B}\nu(x)d\mu(x)\Big)^{p}\Big({1\over \mu(B)}\int_{B}\lambda_2(x)d\mu(x)\Big) \\
&\lesssim  {\nu(B)^p\lambda_2(B)\over \mu(B)^{1+p}} .
\end{align*}
This implies that
$${1\over \nu(B)}{\mu(B)\over \lambda_2(B)^{1\over p}}\lambda_1(B)^{1\over p}\lesssim1,$$
and hence, together with \eqref{lower bd1}, gives
\begin{align*}
{1\over \nu(B)}\int_{B}|b(x)-b_B|d\mu(x)
\lesssim  \|  {C}_b : L^p_{\lambda_1}(X) \rightarrow L^p_{\lambda_2}(X) \|.
\end{align*}

Therefore, $b\in {\rm BMO}_{\nu }(X)$, and
\begin{equation*}
 \| b \| _{{\rm BMO}_{\nu }(X)} \leq C  \| {C}_b : L^p_{\lambda_1}(X) \rightarrow L^p_{\lambda_2}(X) \|.
\end{equation*}

The proof of Theorem \ref{thm main3} is complete.
\end{proof}

\bigskip
\bigskip

{\bf Acknowledgments:}
R.M. Gong is supported by the State Scholarship Fund of China (No. 201908440061).
Q.Y. Wu is supported by the Natural Science Foundation of China
(Grant No. 12091197) and the Natural Science Foundation of Shandong
Province (Grant Nos. ZR2019YQ04 and 2020KJI002).

\medskip

Ruming Gong, School of Mathematical Sciences, Guangzhou University, Guangzhou, China.

\smallskip

{\it E-mail}: \texttt{gongruming@gzhu.edu.cn}

\vspace{0.3cm}

Qingyan Wu, Department of Mathematics, Linyi University,
         Shandong, 276005, China.

\smallskip

{\it E-mail}: \texttt{wuqingyan@lyu.edu.cn}

\vspace{0.3cm}

Manasa N. Vempati,
Department of Mathematics,
Washington University in St. Louis,
One Brookings Drive,
St. Louis, MO USA 63130-4899.

\smallskip
{\it E-mail}: \texttt{m.vempati@wustl.edu}

\end{document}